\documentclass[12pt]{article}
\usepackage{amsfonts,amssymb,amsmath,amsthm}
\usepackage[english]{babel}
\date{}

\theoremstyle{plain}
\newtheorem{theorem}{Theorem}

\newtheorem{lemma}{Lemma}

\theoremstyle{definition}
\newtheorem{definition}{Definition}

\numberwithin{equation}{section}
\numberwithin{theorem}{section}
\numberwithin{proposition}{section}
\numberwithin{lemma}{section}
\numberwithin{corollary}{section}
\numberwithin{definition}{section}
\numberwithin{remark}{section}
\newcommand{\R}{\mathbb{R}}
\newcommand{\N}{\mathbb{N}}
\newcommand{\Z}{\mathbb{Z}}
\newcommand{\Q}{\mathbb{Q}}
\newcommand{\D}{\mathcal{D}}
\newcommand{\esslim}{\operatornamewithlimits{ess\,lim}}
\newcommand{\sign}{\operatorname{sign}}
\newcommand{\meas}{\operatorname{meas}}
\newcommand{\Cl}{\operatorname{Cl}}
\newcommand{\Int}{\operatorname{Int}}
\newcommand{\const}{\mathrm{const}}
\renewcommand{\div}{\operatorname{div}}

\begin{document}
\title{On decay of entropy solutions to multidimensional conservation laws}

\author{Evgeny Yu. Panov}
\maketitle
\begin{abstract}
Under a precise genuine nonlinearity assumption we establish the decay of entropy solutions of a multidimensional scalar conservation law with merely continuous flux.
\end{abstract}

\section{Introduction}
In the half-space $\Pi=\R_+\times\R^n$, $\R_+=(0,+\infty)$, we consider a first order multidimensional
conservation law
\begin{equation}\label{1}
u_t+\div_x\varphi(u)=0,
\end{equation}
where the flux vector $\varphi(u)$ is supposed to be only continuous: $\varphi(u)=(\varphi_1(u),\ldots,\varphi_n(u))\in C(\R,\R^n)$.
Recall the notion of entropy solution to the Cauchy problem for equation (\ref{1}) with initial condition
\begin{equation}\label{2}
u(0,x)=u_0(x)\in L^\infty(\R^n)
\end{equation}
in the sense of S.N.~Kruzhkov \cite{Kr}.

\begin{definition}\label{def1}
A bounded measurable function $u=u(t,x)\in L^\infty(\Pi)$ is called an entropy solution (e.s. for
short) of (\ref{1}), (\ref{2}) if for all $k\in\R$
\begin{equation}\label{entr}
|u-k|_t+\div_x[\sign(u-k)(\varphi(u)-\varphi(k))]\le 0
\end{equation}
in the sense of distributions on $\Pi$ (in $\D'(\Pi)$), and
\begin{equation}\label{ini}
\esslim_{t\to 0+} u(t,\cdot)=u_0 \mbox{ in } L^1_{loc}(\R^n).
\end{equation}
\end{definition}
Condition (\ref{entr}) means that for all non-negative test functions $f=f(t,x)\in C_0^1(\Pi)$
$$
\int_\Pi [|u-k|f_t+\sign(u-k)(\varphi(u)-\varphi(k))\cdot\nabla_xf]dtdx\ge 0
$$
(here ``$\cdot$'' denotes the inner product in $\R^n$).

As was established in \cite[Corollary~7.1]{PaJHDE}, after possible correction on a set of null
measure, an e.s. $u(t,x)$ is continuous on $[0,+\infty)$ as a map $t\mapsto u(t,\cdot)$ into
$L^1_{loc}(\R^n)$. Thus, we may and will always suppose that e.s. satisfy the continuity property $$u(t,\cdot)\in C([0,+\infty),L^1_{loc}(\R^n)).$$
In view of (\ref{ini}), we see that $u(0,x)=u_0(x)$ and in (\ref{ini}) we may replace the essential limit by the usual one.

When the flux vector is Lipschitz continuous, the existence and uniqueness of e.s. to the problem
(\ref{1}), (\ref{2}) are well-known (see \cite{Kr}). In the case under consideration when the flux
functions are merely continuous, the effect of infinite speed of propagation for initial perturbations
appears, which leads even to the nonuniqueness of e.s. to problem (\ref{1}), (\ref{2}) if $n>1$ (see
examples in \cite{KrPa1,KrPa2}).

But, if initial function is periodic (at least in $n-1$ independent directions), the uniqueness holds:
an e.s. of (\ref{1}), (\ref{2}) is unique and space-periodic, see the proof in \cite{PaMax1,PaMax2,PaIzv}. In general case there always exists the unique maximal and minimal e.s., see \cite{ABK,PaMax2,PaIzv}.

The aim of the present paper is the study of the long time decay property of e.s.
for localized in space (in some wide sense) initial data under precise genuine nonlinearity conditions on the flux vector.
In the case of one space variable $n=1$ and strictly convex flux function $\varphi(u)\in C^2(\R)$ the decay of e.s. is well-known, see the book \cite{DafBook} and the references therein.
In general multidimensional setting the decay property was studied mainly for space-periodic e.s. The analytical approach to such study was developed by G.-Q.~Chen and H.~Frid in \cite{ChF}. In particular, in this paper the following decay property for e.s. $u=u(t,x)$ of (\ref{1}), (\ref{2}) was proved
\begin{equation}\label{decp}
\lim_{t\to+\infty} u(t,\cdot)=M,
\end{equation}
where
$$
M=\frac{1}{|P|}\int_{P} u_0(x)dx
$$
is the mean value of initial data. Here
$P$ is a fundamental parallelepiped (a cell of periodicity) for the lattice of periods, and $|P|$ (or, alternatively, $\meas P$) denotes the Lebesgue measure of $P$. The above decay property was proved under rather restrictive regularity and genuine nonlinearity requirements. In   subsequent papers \cite{PaAIHP,Daferm,PaNHM} these requirements were significantly relaxed.

We will need in the sequel results of \cite[Theorem~1.3]{PaNHM}.
Suppose that the initial function $u_0$ is periodic with a lattice of periods $L$, i.e., $u_0(x+e)=u_0(x)$ a.e. on $\R^n$ for every $e\in L$ (we will call such functions $L$-periodic). Denote by $P$ a fundamental parallelepiped for the lattice $L$, and by $L'$ the dual lattice $L'=\{ \ \xi\in\R^n \ | \ \xi\cdot x\in\Z \ \forall x\in L \ \}$. Let, as in (\ref{decp}), $\displaystyle M=\frac{1}{|P|}\int_{P} u_0(x)dx$.

\begin{theorem}\label{thDP}
Suppose that
\begin{eqnarray}\label{NDp}
\forall\xi\in L', \xi\not=0 \ \mbox{ the function } u\to\xi\cdot\varphi(u) \nonumber\\ \mbox{ is not affine on any vicinity of } M.
\end{eqnarray}
Then the decay property (\ref{decp}) holds.
\end{theorem}

In this paper we suppose that the initial function is such that
\begin{equation}\label{van}
\forall\lambda>0 \quad \meas \{ \ x\in\mathbb{R}^n: \ |u_0(x)|>\lambda \ \}<+\infty
\end{equation}
(in particular, this requirement is satisfied for
$u_0\in L^p(\R^n)$, $1\le p<\infty$). Assume also that the merely continuous flux vector $\varphi(u)$ satisfies the following \textit{genuine nonlinearity} requirement
\begin{equation}\label{GN}
\mbox{ the vector } \varphi(u) \mbox{ is not affine on any nonempty interval } (a,b) \mbox{ such that } ab=0.
\end{equation}
To study the decay property, we introduce the topology on $L^\infty(\R^n)$ stronger than one induced by $L^1_{loc}(\R^n)$. This topology is generated by the following norm
\begin{equation}\label{normX}
\|u\|_X=\sup_{y\in\R^n} \int_{|x-y|<1} |u(x)|dx
\end{equation}
(where we denote by $|z|$ the Euclidean norm of a finite-dimensional vector $z$).
Obviously, this norm is shift-invariant: $\|u(\cdot+y)\|_X=\|u\|_X$ for each $y\in\R^n$. It is not difficult to verify that norm (\ref{normX}) is equivalent to each of more general norms
\begin{equation}\label{normV}
\|u\|_V=\sup_{y\in\R^n} \int_{y+V} |u(x)|dx,
\end{equation}
where $V$ is any bounded open set in $\R^n$ (the original norm $\|\cdot\|_X$ corresponds to the unit ball $|x|<1$). For the sake of completeness we prove this result in Lemma~\ref{equ} below.

Our main result is the following decay property.

\begin{theorem}\label{thM}
If $u(t,x)$ is an e.s. of (\ref{1}), (\ref{2}) then, under assumption (\ref{GN}),
\begin{equation}\label{dec}
\lim_{t\to+\infty}\|u(t,\cdot)\|_X=0.
\end{equation}
\end{theorem}

Notice that condition (\ref{GN}) is precise. In fact, if it fails, there exists an interval $(a,0)$ or $(0,b)$ where the vector $\varphi(u)$ is affine. Assume for definiteness that
$\varphi(u)=uc+d$ on an interval $(0,b)$, $b>0$, where $c,d\in\R^n$. If initial function $u_0(x)$ is such that $0\le u_0(x)\le b$ then an e.s. of (\ref{1}), (\ref{2}) is the traveling wave $u(t,x)=u_0(x-tc)$. If $u_0\not=0$ this e.s. does not satisfy (\ref{dec}).

Remark that under assumption (\ref{van}), $u_0(x-tc)\to 0$ as $t\to+\infty$ in $L^1_{loc}(\R^n)$ if the speed $c\not=0$. One of the reason for using the stronger topology of $X$ is to exclude the decay of such traveling waves.

\section{Proof of the main results}
\subsection{Auxiliary lemmas}

\begin{lemma}\label{equ}
The norms $\|\cdot\|_V$ defined in (\ref{normV}) are mutually equivalent.
\end{lemma}

\begin{proof}
Let $V_1,V_2$ be open bounded sets in $\R^n$, and $K_1=\Cl V_1$ be the closure of $V_1$. Then $K_1$ is a compact set while $y+V_2$, $y\in K_1$, is its open covering.
By the compactness there is a finite set $y_i$, $i=1,\ldots,m$, such that
$\displaystyle K_1\subset \bigcup\limits_{i=1}^m (y_i+V_2)$. This implies that for every $y\in\R^n$ and $u=u(x)\in L^\infty(\R^n)$
$$
\int_{y+V_1}|u(x)|dx\le\sum_{i=1}^m \int_{y+y_i+V_2}|u(x)|dx\le m\|u\|_{V_2}.
$$
Hence, $\forall u=u(x)\in L^\infty(\R^n)$
$$
\|u\|_{V_1}=\sup_{y\in\R^n}\int_{y+V_1}|u(x)|dx\le m\|u\|_{V_2}.
$$
Changing the places of $V_1$, $V_2$, we obtain the inverse inequality
$\|u\|_{V_2}\le l\|u\|_{V_1}$ for all $u\in L^\infty(\R^n)$, where $l$ is some positive constant. This completes the proof.
\end{proof}

\begin{lemma}\label{lem1}
Let $X_\alpha$, $\alpha\in\N$, be a countable family of proper linear subspaces of $\R^n$. Then there exists a lattice $L\subset\R^n$ such that $L\cap X_\alpha=\{0\}$ for all $\alpha\in\N$.
\end{lemma}

\begin{proof}
Denote by $\mathrm{L}_n$ the linear space of linear endomorphisms of $\R^n$, and by
$\mathrm{GL}_n$ the group of linear automorphisms of $\R^n$. It is clear that $\mathrm{GL}_n$ is an open subset of $\mathrm{L}_n$, this set can be identified with the set of all $n\times n$ matrices with nonzero determinant. For $\alpha\in\N$, $\xi\in\Z^n$, $\xi\not=0$, we define the sets
$$
H_{\xi,\alpha}=\{ \ A\in\mathrm{L}_n: \ A\xi\in X_\alpha \ \}, \quad
H=\bigcup_{\xi\in\Z^n\setminus\{0\},\alpha\in\N} H_{\xi,\alpha}.
$$
Obviously, the sets $H_{\xi,\alpha}$ are proper linear subspaces of $\mathrm{L}_n$ and therefore they have zero Lebesgue measure in $\mathrm{L}_n$. This implies that
$H$ is a set of zero measure as a countable union of $H_{\xi,\alpha}$. Since the measure of $\mathrm{GL}_n$ is positive (even infinite), then $\mathrm{GL}_n\not\subset H$ and we can find $A\in \mathrm{GL}_n$ such that $A\not\in H$. We define the lattice $L$ as the image of the standard lattice $\Z^n$ under the automorphism $A$: $L=A(\Z^n)$. Since $A\notin H$, we conclude that $L$ satisfies the required condition $L\cap X_\alpha=\{0\}$ $\forall\alpha\in\N$.
\end{proof}

We define the set $F\subset\R$ consisting of points $u_0$ such that the vector $\varphi(u)$ is not affine in any vicinity of $u_0$, and denote $F_+=F\cap (0,+\infty)$, $F_-=F\cap (-\infty,0)$.

\begin{lemma}\label{lem3}
Assume that genuine nonlinearity assumption (\ref{GN}) is satisfied. Then
\begin{equation}\label{Fpm}
\sup F_-=\inf F_+=0.
\end{equation}
\end{lemma}

\begin{proof}
Supposing the contrary, we find that either $\sup F_-<0$ or $\inf F_+>0$. We consider the latter case $\inf F_+>0$, the former case $\sup F_-<0$ is treated similarly.
Let $0<b<\inf F_+$ (notice that $\inf F_+=+\infty$ in the case $F_+=\emptyset$). We see that $(0,b)\cap F=\emptyset$, that is, the vector $\varphi(u)$ is affine in some vicinity of each point in $(0,b)$. Therefore, $\varphi(u)\in C^\infty((0,b))$ and $\varphi'(u)$ is piecewise constant continuous function on $(0,b)$. This is possible only if $\varphi'(u)$ is constant, $\varphi'(u)\equiv c\in\R^n$. This implies that $\varphi(u)=uc+d$ on $(0,b)$, $d\in\R^n$. Hence, the vector $\varphi(u)$ is affine on $(0,b)$, which contradicts to (\ref{GN}).
\end{proof}

\subsection{Proof of Theorem~\ref{thM}}

The proof is relied on the decay property for periodic e.s. First we choose a lattice of periods $L$.

Let $J$ be the sets of intervals $I=(a,b)$ with rational ends $a,b\in\Q$ such that $I\cap F\not=\emptyset$. It is clear that $J$ is a countable set.
For each $I\in J$ we define the linear sets
$$
X_I=\{ \ \xi\in\R^n: \ \mbox{ the function } u\to\xi\cdot\varphi(u) \mbox{ is affine on } I \ \}.
$$
Then $X_I\not=\R^n$, otherwise, the entire vector $\varphi(u)$ is affine on $I$, which contradicts to the condition that $I$ is a neighborhood of some point $u_0\in I\cap F$. Hence $X_I$, $I\in J$, are proper linear subspace of $\R^n$. By Lemma~\ref{lem1} we can find a lattice $L_1$ in $\R^n$
such that $\xi\notin X_I$ for all $\xi\in L_1$, $\xi\not=0$, and all $I\in J$. Let $L=L_1'=\{e\in\R^n: \xi\cdot e\in\Z \ \forall\xi\in L_1\}$ be the dual lattice. Then by the duality $L_1=L'$. By the density of $\Q$, any nonempty interval $(a,b)$ intersecting with F contains some interval $I\in J$. Since every nonzero $\xi\in L'=L_1$ does not belong to $X_I$, we claim that the function $\xi\cdot\varphi(u)$ is not affine on $I$ and, all the more, on $(a,b)$. Hence,
\begin{eqnarray}\label{p1}
\forall\xi\in L', \xi\not=0, \mbox{ the function } u\to\xi\cdot\varphi(u) \nonumber\\
\mbox{ is not affine on intervals intersecting with } F.
\end{eqnarray}
Let $e_k$, $k=1,\ldots,n$, be a basis of the lattice $L$. We define for $r>0$ the parallelepiped
$$
P_r=\left\{ \ x=\sum_{k=1}^n x_ke_k: \ -r/2\le x_k<r/2, k=1,\ldots,n \ \right\}.
$$
It is clear that $P_r$ is a fundamental parallelepiped for a lattice $rL$.
We introduce the functions
$$
v_{0r}^+(x)=\sup_{e\in L} u_0(x+re), \quad v_{0r}^-(x)=\inf_{e\in L} u_0(x+re).
$$
Since $L$ is countable, these function are well-defined in $L^\infty(\R^n)$,
and $\|v_{0r}^\pm\|_\infty\le C_0\doteq\|u_0\|_\infty$. It is clear that $v_{0r}^\pm(x)$ are $rL$-periodic and
\begin{equation}\label{p2v}
v_{0r}^-(x)\le u_0(x)\le v_{0r}^+(x).
\end{equation}
We denote
$$
V_r(x)=\sup_{e\in L} |u_0(x+re)|, \quad M_r=\frac{1}{|P_r|}\int_{P_r} V_r(x)dx.
$$
It is clear that for a.e. $x\in\R^n$
$$
|v_{0r}^\pm(x)|\le V_r(x)\le C_0.
$$
Let us show that under condition (\ref{van})
\begin{equation}\label{l3}
M_r\to 0 \ \mbox{ as } r\to+\infty.
\end{equation}
For that we fix $\varepsilon>0$ and define the set $A=\{ \ x\in\R^n: \ |u_0(x)|>\varepsilon \ \}$.
In view of (\ref{van}) the measure of this set is finite, $\meas A=p<+\infty$. We also define the sets
$$
A_r^e=\{ \ x\in P_r: \ x+re\in A \ \}\subset P_r, \quad r>0, \ e\in L, \quad A_r=\bigcup_{e\in L} A_r^e.
$$
By the translation invariance of Lebesgue measure we have
$$
\sum_{e\in L}\meas A_r^e=\sum_{e\in L}\meas (re+A_r^e)=\sum_{e\in L}\meas (A\cap (re+P_r))=\meas A=p.
$$
This implies that
\begin{equation}\label{Ar}
\meas A_r\le\sum_{e\in L}\meas A_r^e=p.
\end{equation}
If $x\notin A_r$ then $|u_0(x+re)|\le\varepsilon$ for all $e\in L$, which implies that $V_r(x)\le\varepsilon$. Taking (\ref{Ar}) into account, we find
$$
\int_{P_r} V_r(x)dx=\int_{A_r} V_r(x)dx+\int_{P_r\setminus A_r} V_r(x)dx\le C_0\meas A_r+\varepsilon\meas P_r\le C_0p+\varepsilon|P_r|.
$$
It follows from this estimate that
$$\limsup_{r\to+\infty} M_r\le\lim_{r\to+\infty}\left(\frac{C_0p}{|P_r|}+\varepsilon\right)=\varepsilon$$ and since $\varepsilon>0$ is arbitrary, we conclude that (\ref{l3}) holds.
Let
$$
M_r^\pm=\frac{1}{|P_r|}\int_{P_r} v_{0r}^\pm(x)dx
$$
be mean values of $rL$-periodic functions $v_{0r}^\pm(x)$.
In view of (\ref{l3})
\begin{equation}\label{l4}
|M_r^\pm|\le M_r\mathop{\to}_{r\to\infty} 0.
\end{equation}
By (\ref{l4}) and (\ref{Fpm}) we can find such values $B_r^-,B_r^+\in F$, where $r>0$ is sufficiently large, that $B_r^-\le M_r^-\le M_r^+\le B_r^+$, and that $B_r^\pm\to 0$ as $r\to\infty$. We define the $rL$-periodic functions
$$
u_{0r}^+(x)=v_{0r}^+(x)-M_r^++B_r^+\ge v_{0r}^+(x), \quad u_{0r}^-(x)=v_{0r}^-(x)-M_r^-+B_r^-\le v_{0r}^-(x)
$$
with the mean values $B_r^+,B_r^-$, respectively. In view of (\ref{p2v}), we have
\begin{equation}\label{p2}
u_{0r}^-(x)\le u_0(x)\le u_{0r}^+(x).
\end{equation}
Let $u_r^\pm$ be unique (by \cite[Corollary~3]{PaMax1}) e.s. of (\ref{1}), (\ref{2}) with initial functions
$u_{0r}^\pm$, respectively. Taking into account that $(rL)'=\frac{1}{r}L'$, we derive from (\ref{p1}) that condition (\ref{NDp}), corresponding to the lattice $rL$ and the mean value $M=B_r^\pm\in F$, is satisfied. By Theorem~\ref{decp} we find that
\begin{equation}\label{l5}
\lim_{t\to+\infty}\int_{P_r}|u_r^\pm(t,x)-B_r^\pm|dx=0.
\end{equation}
By the periodicity, for each $y\in\R^n$
$$
\int_{y+P_r}|u_r^\pm(t,x)-B_r^\pm|dx=\int_{P_r}|u_r^\pm(t,x)-B_r^\pm|dx,
$$
which readily implies that for $V=\Int P_r$
$$
\|u_r^\pm(t,x)-B_r^\pm\|_V=\int_{P_r}|u_r^\pm(t,x)-B_r^\pm|dx.
$$
In view of Lemma~\ref{equ} we have the estimate
$$\|u_r^\pm(t,x)-B_r^\pm\|_X\le C\int_{P_r}|u_r^\pm(t,x)-B_r^\pm|dx, \ C=C_r=\const.$$
By (\ref{l5}) we claim that
\begin{equation}\label{l5a}
\lim_{t\to+\infty}\|u_r^\pm(t,\cdot)-B_r^\pm\|_X=0.
\end{equation}
Let $u=u(t,x)$ be an e.s. of the original problem (\ref{1}), (\ref{2}) with initial data $u_0(x)$. Since the functions $u_{0r}^\pm$ are periodic, then it follows from (\ref{p2}) and the comparison principle \cite[Corollary~3]{PaMax1} that $u_r^-\le u\le u_r^+$ a.e. in $\Pi$. This readily implies the relation
\begin{eqnarray}\label{l6}
\|u(t,\cdot)\|_X\le \|u_r^-(t,\cdot)\|_X+\|u_r^+(t,\cdot)\|_X\le \nonumber\\
\|u_r^-(t,x)-B_r^-\|_X+\|u_r^+(t,x)-B_r^+\|_X+c(|B_r^-|+|B_r^+|),
\end{eqnarray}
where $c$ is Lebesgue measure of the unite ball $|x|<1$ in $\R^n$.
In view of (\ref{l5a}) it follows from (\ref{l6}) in the limit as $t\to+\infty$
that
$$
\limsup_{t\to+\infty}\|u(t,\cdot)\|_X\le c(|B_r^-|+|B_r^+|).
$$
Since $B_r^\pm\to 0$ as $r\to\infty$, the latter relation implies the desired decay property
$$
\lim_{t\to+\infty}\|u(t,\cdot)\|_X=0.
$$

\section{Acknowledgements}
This work was supported by the Ministry of Science and Higher Education of the Russian  Federation (project no. 1.445.2016/1.4) and by the Russian Foundation for Basic Research (grant 18-01-00472-a.)


\begin{thebibliography}{944}
\bibitem{ABK}
B.\,P.~Andreianov, Ph.~B\'enilan, S.\,N.~Kruzhkov. $L^1$-theory of scalar conservation law with continuous flux function. J. of Functional Analysis.  V. 171 (2000), 15--33.
\bibitem{ChF}
G.-Q.~Chen, H.~Frid. Decay of entropy solutions of nonlinear conservation laws. Arch. Ration. Mech. Anal. 146, No. 2 (1999), 95--127.
\bibitem{Daferm}
C.\,M.~Dafermos. Long time behavior of periodic solutions to scalar conservation laws in several space dimensions. SIAM J. Math. Anal. 45 (2013), 2064--2070.
\bibitem{DafBook}
C.\,M.~Dafermos. Hyperbolic Conservation Laws in Continuum Physics. IV Ed. Grundlehren der mathematischen Wissenschaften, 325, 826 pp. Springer, 2016.
\bibitem{Kr}
S.\,N.~Kruzhkov. First order quasilinear equations in several independent variables. Math. USSR Sb. 10 (1970), 217--243.
\bibitem{KrPa1}
S.\,N.~Kruzhkov, E.\,Yu.~Panov. First-order conservative quasilinear laws with an infinite domain of dependence on the initial data. Soviet Math. Dokl. 42 (1991), 316--321.
\bibitem{KrPa2}
S.\,N.~Kruzhkov, E.\,Yu.~Panov. Osgood's type conditions for uniqueness of entropy solutions to Cauchy problem for quasilinear conservation laws of the first order. Ann. Univ. Ferrara Sez. VII (N.S.) 40 (1994), 31--54.
\bibitem{PaMax1}
E.\,Yu.~Panov. A remark on the  theory of generalized entropy sub- and supersolutions of the Cauchy problem for a first-order quasilinear equation. Differ. Equ. 37 (2001), 272--280.
\bibitem{PaMax2}
E.\,Yu.~Panov. Maximum and minimum generalized entropy solutions to the Cauchy problem for a first-order quasilinear equation. Sb. Math. 193, No.~5 (2002), 727--743.
\bibitem{PaIzv}
E.\,Yu.~Panov. On generalized entropy solutions of the Cauchy problem for a first order quasilinear equation in the class of locally summable functions. Izv. Math. 66, No. 6 (2002), 1171--1218.
\bibitem{PaJHDE}
E.\,Yu.~Panov. Existence of strong traces for generalized solutions of multidimensional scalar conservation laws.
J. Hyperbolic Differ. Equ. 2 (2005), 885--908.
\bibitem{PaAIHP}
E.\,Yu.~Panov. On decay of periodic entropy solutions to a scalar conservation law. Ann I. H.~Poincare-AN. 30 (2013), 997--1007.
\bibitem{PaNHM}
E.\,Yu.~Panov. On a condition of strong precompactness and the decay of periodic entropy solutions to scalar conservation laws. Netw. Heterog. Media. 11 (2016), 349--367.
\end{thebibliography}
\end{document}